\documentclass[12pt]{amsart}
\usepackage{amssymb}

\theoremstyle{plain}
\newtheorem{thm}{Theorem}[section]
\newtheorem{lmm}[thm]{Lemma}

\newtheorem{prp}[thm]{Proposition}
\theoremstyle{definition}

\theoremstyle{remark}
\newtheorem{rem}[thm]{Remark}

\usepackage{url}

\def\SS{\mathbb{S}}
\def\R{\mathbb{R}}
\def\C{\mathbb{C}}
\def\Z{\mathbb{Z}}
\def\T{\mathbb{T}}
\def\N{\mathbb{N}}
\def\cP{{\mathcal{P}}}
\def\cH{{\mathcal{H}}}
\def\cY{{\mathcal{Y}}}
\def\cO{{\mathcal{O}}}
\def\bK{{\mathcal K}}
\def\Dn{N_n^{(d)}}
\def\Span{{\mathrm{Span}}}

\def\sU{{\ U}}
\def\conf{\mathrm{Conf}}
\def\th{\theta}
\def\inlawto{\stackrel{d}{\Rightarrow}}

\def\e{e}
 \def\trivial{\mathbf{1}}
 \def\bra{\langle}
 \def\ket{\rangle}
 \def\wbar{\bar{w}}
 \def\eps{\epsilon}
 \def\ginibre{\mathrm{Ginibre}}
 \def\Span{\mathrm{Span}} 
\def\SS{{\mathbb S}} 
 
\def\F{{\mathcal F}}

\def\la{\lambda}

\def\F21{{_2F_1}}

\numberwithin{equation}{section}

\title{Scaling limit for determinantal point processes \\ on spheres}
\dedicatory{Dedicated to the memory of Professor Yoichiro Takahashi} 
\author{Makoto \textsc{Katori}}
\address{Fakult\"{a}t f\"{u}r Mathematik, Universit\"{a}t Wien, Oskar-Morgenstern-Platz 1,
A-1090 Wien, Austria. On sabbatical leave from
Department of Physics, Faculty of Science and Engineering, Chuo University, Kasuga, Bunkyo-ku, Tokyo 112-8551,
Japan} 
\email{katori@phys.chuo-u.ac.jp}
\author{Tomoyuki \textsc{Shirai}}
\address{Institute of Mathematics for Industry, 
Kyushu University, 744 Motooka, Nishi-ku, Fukuoka 819-0395, Japan}  
\email{shirai@imi.kyushu-u.ac.jp}

\begin{document}
\begin{abstract} 
The unitary group with the Haar probability measure 
is called Circular Unitary Ensemble.  
All the eigenvalues lie on the unit
circle in the complex plane and they can be regarded as a
 determinantal point process on $\SS^1$. 
It is also known that 
the scaled point processes converge weakly to the determinantal point
process associated with the so-called sine kernel 
as the size of matrices tends to $\infty$. 
We extend this result to the case of high-dimensional spheres 
and show that the scaling limit processes are determinantal
 point processes associated with the kernels expressed by the
 Bessel functions of the first kind. 
\end{abstract}
\maketitle

\section{Introduction}\label{sec:intro}

Circular Unitary Ensemble (CUE) is the group $\sU(n)$ of $n
\times n$ unitary matrices with the normalized Haar
measure. 
It is well-known that the joint probability distribution of 
eigenvalues $\{e^{\sqrt{-1} \theta_j}\}_{j=1}^n$ is given by 
\[
p_{\text{CUE}}(\th_1,\th_2,\dots,\th_n)
= \frac{1}{(2\pi)^n n!}  \prod_{1 \le j < k \le n}
	|e^{\sqrt{-1} \theta_j} - e^{\sqrt{-1} \theta_k}|^2
\]
and they form a determinantal point process on $\T = \R / 2\pi \Z$ 
      with $\la(d\theta) = d\theta/(2\pi)$ on $\T$ and 
correlation kernel $K_n$ given by 
\begin{align*}
K_n(\theta, \theta') 
&= \sum_{k=0}^{n-1} 
e^{\sqrt{-1} k \theta} \overline{e^{\sqrt{-1} k \theta'}} 
= v(\theta) \underbrace{\frac{\sin[n (\theta - \theta')/2]}{\sin [(\theta - \theta')/2]}}_{:= \tilde{K}_n(\theta, \theta')} \overline{v(\theta')},  
\end{align*}
where $v(\theta) = e^{\sqrt{-1} (n-1) \theta/ 2}$. 
Since determinantal point processes (DPPs, for short in what
follows) are invariant under a gauge transformation
$K(\theta, \theta') 
\mapsto \tilde{K}(\theta, \theta') := u(\theta) K(\theta, \theta')
\overline{u(\theta')}$ with
$u :\T \to U(1)$, the kernel 
\[
\tilde{K}_n(\theta, \theta') 
= \frac{\sin[n (\theta - \theta')/2]}{\sin[(\theta - \theta')/2]} 
\]
defines the same DPP 
(see Section~\ref{sec:DPP} for the definition and some
invariance properties of DPPs). 
The empirical distribution $\frac{1}{n} \sum_{i=1}^n
\delta_{\theta_i}$ of points converges to the
uniform distribution on $\T$ almost surely. 
On the other hand, scaling as $x_i = n \theta_i \in \R$ 
yields 
\[
\frac{1}{n} \tilde{K}_n\big(\frac{x}{n},
\frac{x'}{n} \big) 
= \frac{1}{n} \frac{\sin[(x - x')/2]}{\sin [(x/n - x'/n)/2]} 
\to \frac{\sin[(x - x')/2]}{(x - x')/2} 
=: K_{\mathrm{sinc}}(x, x'). 
\]
Here we used the term \textit{sinc} instead of
\textit{sine}. 
From this observation, we can see that, as $n \to \infty$,  
\[
\text{$n$-point DPP on $\T$ $\inlawto$ 
the DPP on $\R^1$ with $K_{\mathrm{sinc}}$}, \quad 
\]
where $\inlawto$ means convergence in law. 
In this article, we discuss a generalization of this fact
(see Theorem~\ref{thm:main}): 
under suitable scaling, 
\[
\text{$n$-point DPP on $\SS^d$ $\inlawto$ 
the DPP on $\R^d$ with $K^{(d)}(x,y)$}, 
\]
where 
\[
 K^{(d)}(x,y) = \frac{1}{(2\pi |x-y|)^{d/2}}
 J_{d/2}(|x-y|). 
\]
Here $J_{\alpha}(x)$ is the Bessel function of the first
kind defined by 
\begin{equation}
 J_{\alpha}(x) = \sum_{k=0}^{\infty} \frac{(-1)^k}{k!
 \Gamma(k+\alpha+1)} \Big(\frac{x}{2}\Big)^{2k+\alpha}. 
\label{eq:Bessel-alpha} 
\end{equation}
For $d=1$, the correlation kernel $K^{(1)}$ 
coincides with $K_{\mathrm{sinc}}$ up to constant multiple.  
We also remark that the kernel $K^{(d)}$ is the reproducing kernel for
the so-called generalized Paley-Wiener space, which is the space of 
functions that are the Fourier transforms of $L^2$-functions
whose support is contained in the unit ball centered at the
origin in $\R^d$. 

The organization of this paper is as follows: In Section 2, we
discuss two DPPs on $\SS^2$ and show limit theorems for
these DPPs. In Section 3, we review spherical harmonics on
$\SS^{d}$ and their properties. In Section 4, we introduce 
determinantal point processes on $\SS^{d}$ by using
spherical harmonics and in Section 5 we show their scaled 
processes converge to the determinantal point processes
associated with the Bessel functions of the first kind. 

\section{Two determinantal point processes on $\SS^2$}\label{sec:DPPonS2}

In this section, 
we consider two generalizations of CUE on $\T \simeq \SS^1$ 
to the sphere $\SS^2$. 
First, we regard the joint distribution of CUE eigenvalues 
as 
\[
p_{\text{CUE}}(z_1,z_2,\dots,z_n)
= \frac{1}{(2\pi)^n n!}\prod_{1 \le j < k \le n} \|z_j - z_k\|_{\R^2}^2  \quad (z_j \in \SS^1 \subset \R^2)  
\]
by identifying $e^{\sqrt{-1} \th_j} \in \C$ with 
$z_j = e^{\sqrt{-1} \th_j} \in \SS^1 \subset \R^2$. 
By replacing $\SS^1 \subset \R^2$ with
$\SS^2 \subset \R^3$ on the right-hand side, 
we obtain a generalization of CUE eigenvalues,
which turns out to be a DPP on $\SS^2$. 

Second, as before, we regard the eigenvalues of CUE as the
DPP with 
kernel 
\[
K_n(\theta, \phi) 
= \sum_{k=0}^{n-1} e^{\sqrt{-1} k \theta} \overline{e^{\sqrt{-1} k
       \phi}} 
\]
and $\la(d\theta) = d\theta / (2\pi)$ on $\SS^1$. Here
       $e^{\sqrt{-1} k \theta}$ is an eigenfunction of the Laplacian
       $-\Delta_{\SS^1} = -d^2/d\theta^2$
       corresponding to the eigenvalue $k^2$, i.e., 
\[
 -\Delta_{\SS^1} e^{\sqrt{-1} k \theta} = k^2 e^{\sqrt{-1} k \theta}. 
\]
The kernel $K_n$ defines the projection onto the subspace
spanned by $\{e^{\sqrt{-1} k \theta}\}_{k=0}^{n-1}$ in $L^2(\SS^1)$. 
We can generalize CUE from this point of view by considering
the Laplace-Beltrami operator on $\SS^2$ and its eigenspaces
spanned by the spherical harmonics. 

In Sections~\ref{sec:spherical} and \ref{sec:harmonic}, 
we will see that those two point processes
converge to different DPPs in the bulk.

\subsection{Determinantal point processes}\label{sec:DPP}

In this subsection, we briefly review determinantal point
processes from the viewpoint of transformations. 
See, for instance, \cite{K86} for the theory of point
processes and \cite{HKPV,M75, ST0, ST1,S} for details of the basic
properties of determinantal point processes. 

Let $S$ be a locally compact Hausdorff
space with countable basis and $\la$ be a Radon measure on $S$.  
We denote by $\conf(S)$ the configuration space over $S$, the set of
non-negative integer valued Radon measures, which is equipped with
the topological Borel $\sigma$-field with respect to 
the vague topology. A point process on $S$ is 
a $\conf(S)$-valued random variable $\xi = \xi(\omega)$ on a probability space
$(\Omega, \mathcal{F}, P)$ 
or its probability distribution on $\conf(S)$.
For a locally trace class integral operator $K: L^2(S,\la) \to
L^2(S,\la)$ with $O \le K \le I$, 
there exists uniquely a probability measure on
$\conf(S)$ whose correlation function is given by 
\[
 \rho_n(s_1,\dots,s_n) = \det(K(s_i,s_j))_{i,j=1}^n, 
\]
where $K(x,y)$ is called a correlation kernel. 
This point process is called a determinantal point process associated
with $K$ and $\la$ (or just associated with $K$). 
The measure $\la$ is sometimes called a background measure. 
This situation arises in the reproducing kernel Hilbert
space setting. A Hilbert space $H$ of functions on $S$ is called a
reproducing kernel Hilbert space (RKHS) 
if there is a hermitian kernel $K : S\times
S \to \C$ such that (i) $K(\cdot, s') \in H$ for any $s' \in
S$ and (ii) 
$f(s') = \bra f, K(\cdot, s')\ket_H$ for any $f \in H$. 
When $H$ is realized as a closed subspace of $L^2(S,\la)$, the
kernel $K(s,s')$ defines a projection integral operator $K$ from 
$L^2(S, \la)$ onto $H$ so that one can define for such $H$
(and $K$) a DPP on $S$. In this sense, for example, the Ginibre point
process is the DPP associated with the Bargmann-Fock space,
which is the space of square-integrable analytic functions
on $\C$ with respect to the complex Gaussian
measure. The theory of reproducing kernel Hilbert spaces 
can be found in \cite{A50}
and several DPPs are discussed in this context
(cf. \cite{AGR19, BQ17, S16}). 

A pair of kernel and background measure of a DPP is not
unique. 
For example, the following changes of kernels and measures 
do not change the law of DPP. 

\noindent
(i) (Gauge transformation). For a non-vanishing measurable function 
$u : S \to \C$, 
\begin{equation}
 K(s,s') \mapsto \tilde{K}(s,s') := u(s) K(s,s')
  u(s')^{-1}. 
\label{eq:transformation1}
\end{equation}
(ii) For a measurable function $g : S \to [0,\infty)$, 
\begin{equation}
 (K(s,s'), g(s) \la(ds)) \mapsto (\sqrt{g(s)} K(s,s')
 \sqrt{g(s')}, \la(ds)).  
\label{eq:transformation2}
\end{equation}
For example, 
in the case of $S = \C$, the corresponding DPP 
associated with $K(z,w) = e^{z \wbar}$ and $\la(dz) =\pi^{-1}
e^{-|z|^2} dz$ is called a
Ginibre point process. From \eqref{eq:transformation2}, this DPP is also expressed as a DPP 
associated with $\tilde{K}(z,w) = e^{z \wbar - (|z|^2 + |w|^2)/2}$
and $\tilde{\la}(dz) =
\pi^{-1} dz$. 

For a one-to-one measurable mapping $h : X \to Y$ and 
a DPP $\xi = \sum_j \delta_{x_j}$ on $X$ associated with $K$ and
$\la$, the transformed point process 
$\tilde{\xi} = \sum_j \delta_{h(x_j)}$ turns out to be the DPP
on $Y$ associated with 
\begin{equation}
 \tilde{K}(y,y') = K(h^{-1}(y), h^{-1}(y')), \quad 
\tilde{\la}(dy) = (\la \circ h^{-1})(dy). 
\label{eq:transformation3} 
\end{equation}
By using this formula, we can easily see that 
if $\xi = \sum_j \delta_{x_j}$ is the DPP on $X=\R^d$ associated with
$K$ and $dx$, then the scaled point process 
$S_c(\xi) := \sum_j \delta_{c x_j} \ (c>0)$ is the DPP associated
with  
\begin{equation}
 K_c(y,y') := \frac{1}{c^d} K(\frac{y}{c}, \frac{y'}{c}),
 \quad \la(dy) = dy. 
\label{eq:transformation4} 
\end{equation}
Here we used \eqref{eq:transformation3} with $h(x)=cx$ and then
\eqref{eq:transformation2} with $g(x) \equiv c^{-d}$.  

We often use the fact that the measure $K(s,s) \la(ds)$ is
the mean measure of the DPP on $S$ associated with $K$ and $\la$, 
which describes the density of points of the DPP on $S$.

\subsection{Spherical ensemble}\label{sec:spherical}

The Ginibre ensemble is defined as a random matrix
$G=(G_{ij})_{i,j=1}^N$ whose
elements are i.i.d. and $G_{ij} \sim N_{\C}(0,1)$, which is the complex
 standard normal distribution. 
We write it 
\[
G \sim \ginibre(N) \Longleftrightarrow \text{$\{G_{ij}\}_{i,j=1}^N$ are
 i.i.d. and $G_{ij} \sim N_{\C}(0,1)$}. 
\]

Let $A, B \sim \ginibre(N)$ be independent. 
Krishnapur \cite{K09} showed that the eigenvalues of $A^{-1} B$ form a DPP on $\C$
with the following kernel and background measure
\[
K_N(z,w) = (1+z \wbar)^{N-1}, \quad 
\la(dz) = \frac{N}{\pi (1+|z|^2)^{N+1}} dz. 
\]
The corresponding reproducing kernel Hilbert space is 
the space of polynomials of degree $\le N-1$: 
\[
 H_{K_N} = \Span \{z^k : k=0,1,\dots,N-1\}
\]
in $L^2(\C, \la)$. 
The mean measure of the DPP is $K_N(z,z)\la(dz) = \frac{N}{\pi
(1+|z|^2)^2} dz$, which is the push-forward of $N$-times
the uniform measure on the Riemann sphere $\hat{\C} = \C
\cup \{\infty\}$ to $\C$. This fact 
suggests that it is natural to view this DPP as a point
process on $\hat{\C} \simeq \SS^2$ through the stereographic projection. 
The joint distribution with respect to the surface measure is given by 
\[
(\mathrm{const.}) \prod_{1 \le j < k \le N} \|u_j - u_k\|^2_{\R^3} 
\quad \text{on $\hat{\C} \simeq \SS^2$}. 
\]
This point process is again a DPP on $\SS^2$ and is called
the \textit{spherical ensemble} \cite{AZ15, K09}. 
This DPP is clearly 
$O(3)$-invariant, and uniformly distributed with density $N/4\pi$. 
This may be considered as a spherical version of CUE eigenvalues. 
The correlation kernel as DPP is given by 
\begin{align}
 K_N(u,u') 
&= K_N((\theta, \phi), (\theta', \phi')) \nonumber \\
&= \frac{N}{4\pi} \Big(
e^{\sqrt{-1}(\phi-\phi')} \sin(\th/2) \sin(\th'/2) + 
\cos(\th/2) \cos(\th'/2) \Big)^{N-1}
\label{eq:KNspherical}
\end{align}
with respect to the surface measure $\sigma(d\theta d\phi) = \sin
\theta d\theta d\phi$ on $\SS^2$, 
where $(\theta, \phi)$ is the polar coordinates of
$\SS^2$.  
As $N \to \infty$, the empirical measure $\frac{1}{N} 
\sum_{i=1}^N \delta_{u_i}$ converges
      weakly to the uniform measure on $\SS^2$ almost
      surely. 
We define a point process on the tangent space
$T_{e_3}(\SS^2)$ at the north pole $e_3 = (0,0,1)$ as 
the pullback of points on the sphere by 
the exponential map $\exp : T_{e_3}(\SS^2) \to \SS^2$, i.e.,
      using the polar coordinates $(\theta, \phi)$, 
\[
T_{e_3}(\SS^2) \ni (\th \cos \phi, \th \sin \phi) \mapsto 
(\sin \th \cos \phi, \sin \th \sin \phi, \cos \th) \in \SS^2. 
\]
We also identify $(\th \cos \phi, \th \sin \phi) \in 
T_{e_3}(\SS^2)$ with $\th e^{\sqrt{-1} \phi} \in \C$.  

Let $\xi_N$ be the spherical ensemble, 
which is an $N$-point DPP on $\SS^2$. 
Since this DPP is rotation invariant, it suffices to look at the
north pole $e_3 = (0,0,1)$. 
Let $T_{e_3}(\SS^2)$ be the tangent space at $e_3$. 
For fixed $\epsilon > 0$, we consider the pullback of
      points in $\xi_N$ on $\SS^2 \cap B_{\epsilon}(e_3)$ by the 
      exponential map $\exp : T_{e_3}(\SS^2) \to \SS^2$ and
      denote it by $\eta_N^{(\eps)}$, i.e., 
\[
 \eta_N^{(\eps)} := \sum_{u \in \xi_N : u \in
 B_{\eps}(e_3)} 
\delta_{\exp^{-1}(u)}, 
\] 
where $u \in \xi$ means $\xi(\{u\})=1$ and 
$B_{\eps}(e_3)$ is the $\eps$-neighborhood of $e_3$. 
For a configuration $\eta = \sum_j 
      \delta_{x_j} \in \conf(\R^d)$, 
we define a scaling map $S_c :
      \conf(\R^d) \to \conf(\R^d)$ by
\[
 S_c(\eta) = \sum_j \delta_{c x_j} \quad (c > 0). 
\]
In the rest of Section 2, we use this scaling map with $d=2$.
It follows from \eqref{eq:transformation4} 
that the scaled point process $S_c(\eta_N^{(\eps)})$
is the DPP associated with the kernel 
\[
K_{N,c, \epsilon}(x,y) = 
\frac{1}{c^2} K_N(u_c, v_c)\trivial_{B_{\epsilon}}(u_c)
 \trivial_{B_{\epsilon}}(v_c) \quad x, y \in
 T_{e_3}(\SS^2),  
\]
where $u_c = \exp(x/c)$ and $v_c =\exp(y/c)$, 
$B_{\epsilon} = B_{\epsilon}(e_3)$ and
$\trivial_{B_{\epsilon}}$ is its indicator function. 
For weak convergence of DPPs as $N, c \to \infty$, it suffices to show the
uniform convergence of the kernel $K_{N,c,\epsilon}$ on
each compact set $C \subset (T_{e_3}(\SS^2))^2$
(cf. Proposition 3.10 in \cite{ST1}). 
The function $\trivial_{B_{\epsilon}}(u_c)$ 
is eventually $1$ on every compact $C$ 
if $c$ is large enough so that we only
need to check the convergence of 
$K_{N,c}(x,y) := c^{-2} K_N(u_c, v_c)$.

Then we have the following proposition.  
\begin{prp}\label{prp:Ginibre}
The scaled point process $S_{\sqrt{N}}(\eta_N^{(\eps)})$ converges 
weakly to the Ginibre DPP with density $1/4\pi$. Here 
the Ginibre DPP with density $\rho$ is 
the DPP on $\C$ associated with the kernel and background
 measure 
 \[
 K(z,w) = e^{\pi \rho z \wbar}, \quad \la(dz) = \rho e^{-\pi
 \rho |z|^2} dz. 
 \]
\end{prp}
\begin{proof}[Proof of Proposition~\ref{prp:Ginibre}]
We put 
$\theta = r N^{-1/2}$ and $\theta' = r' N^{-1/2}$. 
As $N \to \infty$, we see that 
\begin{align*}
\lefteqn{e^{\sqrt{-1}(\phi-\phi')} \sin(\th/2) \sin(\th'/2) + 
\cos(\th/2) \cos(\th'/2)} \\ 
&\sim 
e^{\sqrt{-1}(\phi-\phi')} \frac{1}{4} \th \th' + 
1 - \frac{1}{8} (\th^2 + (\th')^2) \\
&= 1 + \frac{1}{4N} \{z\wbar - \frac{1}{2} (|z|^2 +
 |w|^2)\},  
\end{align*}
where $z = r e^{\sqrt{-1} \phi}$ and $w = r' e^{\sqrt{-1}
 \phi'}$. Therefore, from \eqref{eq:KNspherical}, we have 
\begin{align*}
K_{N, \sqrt{N}}(z,w)
&= \frac{1}{N} K_N(\exp(\frac{z}{\sqrt{N}}),
 \exp(\frac{w}{\sqrt{N}})) \\ 
&\sim \frac{1}{4\pi} 
\Big(1 + \frac{1}{4N} \{z \wbar - \frac{1}{2} (|z|^2 + |w|^2)\} \Big)^{N-1} \\
&\to \frac{1}{4\pi} e^{\frac{1}{4} \{z\wbar - \frac{1}{2} (|z|^2 + |w|^2)\}}. 
\end{align*}
This together with \eqref{eq:transformation2} completes the proof. 
\end{proof}

\begin{rem}
Spherical ensemble has also been studied as the $2D$
one-component plasma (OCP) or $2D$
Coulomb gas (log-gas) on $\SS^2$ (see, for instance,
\cite{Ca81} and Section 15.6 in \cite{For10}).
Proposition~\ref{prp:Ginibre} 
was observed in \cite{Ca81} in the context of
$2D$ OCP on a sphere,  
where the assertion is stated as follows: the thermodynamic limit of the OCP on the sphere is identical to that of the OCP on the plane. 
\end{rem}

\subsection{DPP associated with the spherical harmonics on $\SS^2$}\label{sec:harmonic}

Let us consider the Laplace-Beltrami operator 
$-\Delta_{\SS^2}$ on $L^2(\SS^2)$. The spectrum consists of the discrete
eigenvalue $\ell(\ell+1)$ with multiplicity $2\ell+1$ for
$\ell =0,1,\dots$. 
We denote by $E_{\ell}$ the eigenspace corresponding to
the eigenvalue $\ell(\ell+1)$ 
and then there is a spectral decomposition of $L^2(\SS^2)$ as 
\[
       L^2(\SS^2) \simeq \bigoplus_{\ell=0}^{\infty}
       E_{\ell}. 
\]
Each eigenspace $E_{\ell}$ is spanned by the so-called spherical
harmonics given by 
\begin{equation}
 Y_{m}^{\ell}(\theta, \phi) := \sqrt{\frac{2\ell+1}{4\pi}
      \frac{(\ell-m)!}{(\ell+m)!}} P_m^{\ell}(\cos \theta)
      e^{\sqrt{-1} m\phi} \quad (-\ell \le m \le \ell), 
\label{eq:Yml}
\end{equation}
where $P_m^{\ell}(x)$ is the associated Legendre polynomial
of degree $m$, i.e.,  
\[
 E_{\ell} = \Span\{Y_m^{\ell} : m=-\ell, -\ell+1, \dots, \ell\}. 
\]

Now we consider the subspace corresponding to the eigenvalues
up to $n(n+1)$, that is  
       $E_{\le n} := \oplus_{\ell=0}^{n} E_{\ell}$. 
The reproducing kernel for $E_{\le n}$ is given by 
\begin{align*}
 K_n(x,y) 
&= \sum_{\ell=0}^{n} 
\sum_{m=-\ell}^{\ell} Y_m^{\ell}(x)
       \overline{Y_m^{\ell}(y)}
=: \sum_{\ell=0}^{n} Z_{\ell}(x,y), 
\end{align*}
where $Z_{\ell}(x,y)$ is the reproducing kernel for
       $E_{\ell}$. 

We consider the DPP $\xi_n$ on $\SS^2$ associated with
$K_n(x,y)$ and the 
Lebesgue measure on $\SS^2$. 
The number of points of DPP $\xi_n$ on $\SS^2$ 
is $(n+1)^2$. 
As $n \to \infty$, the empirical measure converges
      weakly to the uniform measure on $\SS^2$. 
For this DPP, we consider the same problem as in the
previous subsection. For fixed $\epsilon > 0$, let 
       $\eta_n^{(\eps)}$ be the pullback of
       $\xi_n$ on $\SS^2 \cap B_{\epsilon}(e_3)$ by the 
       exponential map $\exp : T_{e_3}(\SS^2) \to \SS^2$. 
Here we have $(n+1)^2$-points and hence the scaling $S_n$ makes
the density of points on $T_{e_3}(\SS^2)$ of $O(1)$. 

\begin{thm}
The scaled point process $S_n(\eta_n^{(\eps)})$ converges 
weakly to the DPP on $T_{e_3}(\SS^2) \simeq \R^2$ 
associated with the kernel and background measure 
\[
K(x,y) = \frac{1}{2\pi |x-y|} J_1(|x-y|), \ \la(dx) = dx, 
\]
where $J_1(r)$ is the Bessel function of the first kind
 \eqref{eq:Bessel-alpha} with $\alpha=1$. 
\end{thm}

The proof is given in Section~\ref{sec:infinite_DPP} 
in higher dimensional setting. 

\section{Spherical harmonics on $\SS^d$}\label{sec:spherical_harmonics}

In this section, we recall basic properties of spherical
harmonics on $\SS^d$ (cf. \cite{Nom18}). 

For $d \in \N$, let
$\cP:=\cP(\R^{d+1})$ be a vector space of all complex-valued
polynomials on $\R^{d+1}$,
and $\cP_{\ell}, \ell \in \Z_{\ge 0} := \{0,1,2, \dots \}$, be its subspaces
consisting of homogeneous polynomials of degree $\ell$;
$p(x)=\sum_{|\alpha|=\ell} c_{\alpha} x^{\alpha}$,
$c_{\alpha} \in \C, x \in \R^{d+1}$.
The vector space of all harmonic functions in $\cP$ is
denoted by $\cH :=\{p \in \cP : \Delta p=0\}$
and let $\cH_{\ell} := \cH \cap \cP_{\ell}$ for $\ell \in \Z_{\ge 0}$.

Now we consider a unit sphere in $\R^{d+1}$ denoted by 
$\SS^d$, in which we use  the polar coordinates 
for $u=(u_1, \dots, u_{d+1}) \in \SS^d$, 
\begin{align}
u_1 &= \sin \theta_d \cdots \sin \theta_2 \sin \theta_1,
\nonumber\\
u_k &= \sin \theta_d \cdots \sin \theta_k \cos \theta_{k-1},
\nonumber\\
u_{d+1} &= \cos \theta_d,
\nonumber\\
& \qquad \theta_1 \in [0, 2 \pi), \quad
\theta_k \in [0, \pi], \quad k=2, \dots, d.
\label{eqn:coordinate1}
\end{align}
Note that $|u|^2 := \sum_{j=1}^{d+1} u_j^2 =1$. 
The standard measure on $\SS^d$ is then given as
\begin{equation}
\sigma_d (du)
:= \sin^{d-1} \theta_{d} \sin^{d-2} \theta_{d-1}
\cdots \sin \theta_2 d \theta_1 \cdots d \theta_{d},
\quad u \in \SS^d.
\label{eqn:sigma}
\end{equation}
The total measure of $\SS^d$ is calculated as
\begin{equation}
\omega_d := \sigma_d(\SS^d)
= \frac{2 \pi^{(d+1)/2}}{\Gamma((d+1)/2)}.
\label{eqn:omega}
\end{equation}
We write the restriction of harmonic polynomials in $\cH_{\ell}$ on
$\SS^d$ as
\[
\cY_{\ell} := \left\{ h\big|_{\SS^d} : h \in \cH_{\ell} \right\},
\quad \ell \in \Z_{\ge 0}.
\]
It is known that 
\begin{equation}
d_{\ell} := \dim \cY_{\ell}
= \frac{2\ell+d-1}{d-1} 
{\ell + d -2 \choose \ell}. 
\label{eqn:dim}
\end{equation}

The reproducing kernel $Z_{\ell}(u,v), u, v \in \SS^d$ is
defined as a unique function in $\cY_{\ell}$ such that
the following reproducing properties hold,
\[ 
	Y(u) = \int_{\SS^d} Z_{\ell}(u,v) Y(v) \mu_d(dv)
\quad \forall Y \in \cY_{\ell}, 
\]
where $\mu_d = \omega_d^{-1} \sigma_d$ 
is normalized to be a probability measure on $\SS^d$. 
Consider an orthonormal basis $\{Y^{\ell}_j \}_{j=1}^{d_{\ell}}$ of $\cY_{\ell}$
with respect to $\mu_d$;
\begin{equation}
\int_{\SS^d} Y^{\ell}_n(u) \overline{Y^{\ell}_m(u)} \mu_d(du) 
=\delta_{nm}, \quad n, m \in \Z_{\ge 0}.
\label{eqn:orthonormal}
\end{equation}
We note that this normalization is different from
\eqref{eq:Yml} up to constant multiple for $d=2$. 
The reproducing kernel $Z_{\ell}$ of $\cY_{\ell}$ in 
$L^2(\SS^d, \mu_d)$ is expanded as
\begin{equation}
Z_{\ell}(u,v) 
= \sum_{j=1}^{d_{\ell}} Y^{\ell}_j(u) \overline{Y^{\ell}_j(v)}.
\label{eqn:Psi1}
\end{equation}
For $\lambda > -1/2$, we define the ultraspherical
polynomial by 
\begin{equation}
P^{\lambda}_{\ell}(s)
:= \F21 \left( -\ell, \ell+2 \lambda; \lambda+ \frac{1}{2} ; 
\frac{1-s}{2} \right),
\label{eqn:P1}
\end{equation}
where $\F21$ denotes the Gauss hypergeometric function,
\[
\F21(\alpha, \beta; \gamma; z)
:= \sum_{n=0}^{\infty} \frac{(\alpha)_n (\beta)_n}{(\gamma)_n}
\frac{z^n}{n!},
\]
with 
$(\alpha)_n := \alpha(\alpha+1) \cdots (\alpha+n-1)
= \Gamma(\alpha+n)/\Gamma(\alpha)$ and $(\alpha)_0 := 1$. 
Then the following equality is established,
\begin{equation}
Z_{\ell}(u,v) 
=d_{\ell} P^{(d-1)/2}_{\ell}(u \cdot v),
\quad u, v \in \SS^d,
\label{eqn:Psi2}
\end{equation}
where $d_{\ell}$ is given by (\ref{eqn:dim}). 
It is clear that $Z_{\ell}$ is $O(d+1, \R)$-invariant in the sense
that 
\begin{equation}
Z_{\ell}(gu,gv) = Z_{\ell}(u,v)
\quad \forall g \in O(d+1, \R), \quad
\forall u, v \in \SS^d.
\label{eqn:invariance}
\end{equation}
If we define the stabilizer subgroup $L_0$ of
$SO(d+1, \R)$ at $\e_{d+1}$ as 
\[
L_0 = \left\{
{\small 
\begin{pmatrix}
A & \ \textbf{0} \\
{^{t}\textbf{0}} & \ 1 
\end{pmatrix}
}
: A \in SO(d, \R) \right\},
\]
and let
\[
\cY_{\ell}^{L_0} :=
\{ Y \in \cY_{\ell} : Y(g u)=Y(u), \quad
\forall g \in L_0, \quad \forall u \in \SS^d \},
\]
then $Z_{\ell}(\cdot, \e_{d+1}) \in \cY_{\ell}^{L_0}$.
The space $\cY_{\ell}^{L_0}$ is the one-dimensional vector space
spanned by $P_{\ell}^{(d-1)/2}(u \cdot \e_{d+1})$.
In general, any $L_0$-invariant function is a constant
on each $L_0$-orbit 
$\cO_s :=\{u \in \SS^d : u \cdot \e_{d+1} = s\}$ for $s \in [-1, 1]$, 
and hence functions in $\cY_{\ell}^{L_0}$ is called
the zonal harmonics of degree $\ell$. 
In the next section, 
we also use the following Gegenbauer polynomial of degree
$\ell$ defined by 
\begin{equation}
C_{\ell}^{\lambda}(s):= \binom{\ell+2 \lambda-1}{\ell} 
P^{\lambda}_{\ell}(s).   
\label{eqn:Gegenbauer}
\end{equation}
The generating function of $C_{\ell}^{\la}(s)$ is given by 
 \[
  (1-2sz + z^2)^{-\la} = \sum_{\ell=0}^{\infty} 
 C_{\ell}^{\la}(s) z^{\ell}. 
 \]

\section{Harmonic ensembles on $\SS^d$}
\label{sec:finite_DPP}
In Section~\ref{sec:DPPonS2}, we discussed 
two DPPs on $\SS^2$ as generalizations of CUE eigenvalues on
$\SS^1$ and showed limit theorems for those DPPs. 
In this section, we consider DPPs on $\SS^d$ sometimes 
called \textit{harmonic ensembles} (cf. \cite{BMOC16}),
which generalize the DPP treated in Section~\ref{sec:harmonic}

For $d \in \N$, $n \in \Z_{\ge 0}$ and $u, v \in \SS^d$, we set 
\begin{align}
\bK^{(d)}_n(u, v)
&:= \sum_{\ell=0}^{n} Z_{\ell}(u,v)
= \sum_{\ell=0}^{n} d_{\ell} P^{(d-1)/2}_{\ell}(u \cdot v)
\nonumber\\
&= \sum_{\ell=0}^{n} \frac{2\ell+d-1}{d-1} C^{(d-1)/2}_{\ell}(u \cdot v).
\label{eqn:K1}
\end{align}
Since $\bK^{(d)}_n(u, v)$ is the reproducing kernel for 
$\cY_{\le n} := \oplus_{\ell=0}^n \cY_{\ell}$ whose 
dimension is $\Dn := \sum_{\ell=0}^{n} d_{\ell}$ with
respect to the background measure $\mu_d$, we have 
a DPP $\xi_n$ on $\SS^d$ such that the correlation kernel is
given by (\ref{eqn:K1}), which is called the harmonic
ensemble on $\SS^d$. The harmonic ensemble is rotation invariant and 
thus the density is uniform on $\SS^d$ given by 
\begin{align*}
\bK^{(d)}_n(u, u) 
&=\sum_{\ell=0}^{n} d_k P^{(d-1)/2}_{\ell}(1)
= \Dn, 
\end{align*}
where we have used the fact that $P^{\lambda}_{\ell}(1) = 1$
for $\la> -1/2$. 
It is easy to see that 
\begin{equation}
 \Dn = \frac{d+2n}{d} {d+n-1 \choose
 n} = \frac{2}{d!} n^d + o(n^d). 
\label{eqn:Dim}
\end{equation}
We note that the number of points in $\xi_n$ is 
equal to $\Dn$ almost surely. 

If we introduce the Jacobi polynomials defined as 
\[
 P_n^{(\alpha, \beta)}(x) 
= \frac{(\alpha+1)_n}{n!} \F21(-n, n + \alpha + \beta + 1;
 \alpha+1 ; \frac{1-x}{2}), 
\]
and 
\begin{equation}
 Q_n^{(\alpha, \beta)}(x) 
= \frac{P_n^{(\alpha, \beta)}(x)}{P_n^{(\alpha, \beta)}(1)}, 
\label{eq:Qndef} 
\end{equation}
then we have the following. 
\begin{lmm} 
For $n \in \Z_{\ge 0}$ and $\la > -1/2$, 
\begin{equation}
\sum_{\ell=0}^n \frac{\ell +\la}{\la}
 C^{\la}_{\ell}(x)
 = \frac{2n + 2\la +1}{2\la +1} {2\la + n \choose n} 
   Q_n^{(\la+1/2, \la-1/2)}(x). 
\label{eq:gegen2jacobi}
\end{equation}
\end{lmm}
\begin{proof}
By using the contiguous relation (Eq.(18.9.5) in
 \cite{NIST10})
\[
(2n+\alpha+\beta+1) P_n^{(\alpha,\beta)}(x) 
= (n+\alpha+\beta+1) 
P_n^{(\alpha,\beta+1)}(x)  
+ (n+\alpha) P_{n-1}^{(\alpha,\beta+1)}(x)
\]
together with $P_n^{(\alpha,\beta+1)}(1)/P_{n-1}^{(\alpha,\beta+1)}(1) =
 (n+\alpha)/n$, we have 
\begin{equation}
(2n+\alpha+\beta+1) Q_n^{(\alpha,\beta)}(x) 
= (n+\alpha+\beta+1) 
Q_n^{(\alpha,\beta+1)}(x)  
+ n Q_{n-1}^{(\alpha,\beta+1)}(x). 
\label{eq:Qn}
\end{equation}
From the recurrence relation of the Gegenbauer polynomials
(see, Eq.(18.9.7) in \cite{NIST10}) and the definition, 
\[
(n+\lambda) C^{\lambda}_n(x)
=\lambda \{C^{\lambda+1}_n(x)-C^{\lambda+1}_{n-2}(x)\} \ (n
 \ge 2), \quad 
C^{\lambda}_0(x)=1, \ C^{\lambda}_1(x)=2 \lambda x, 
\]
we see that 
\begin{align*}
\sum_{\ell=0}^n \frac{\ell +\la}{\la}
 C^{\la}_{\ell}(x)
&= C^{\la}_0(x)+\frac{1 + \la}{\la} C^{\la}_1(x) 
+ \sum_{\ell=2}^n \Big\{ C^{\la+1}_{\ell}(x) 
- C^{\la+1}_{\ell-2}(x) \Big\} \\
 &= 
 C^{\la+1}_{n}(x) + C^{\la+1}_{n-1}(x).  
\label{eqn:K1_2}
\end{align*}
The Gegenbauer polynomial can also be written as 
\[
 C_n^{\alpha+1/2}(x) = 
{n+2\alpha \choose n}
Q_n^{(\alpha, \alpha)}(x). 
\]
From \eqref{eq:Qn} with $\alpha = \beta+1$, 
we obtain 
\begin{align*}
C^{\alpha+1/2}_n(x) + C^{\alpha+1/2}_{n-1}(x) 
&= 
{n+2\alpha \choose n} Q_n^{(\alpha, \alpha)}(x)
+ {n+2\alpha-1 \choose n-1} Q_{n-1}^{(\alpha, \alpha)}(x)\\
&= 
{n+2\alpha-1 \choose n} 
\Big\{\frac{n+2\alpha}{2\alpha} Q_n^{(\alpha, \alpha)}(x)
+ \frac{n}{2\alpha} Q_{n-1}^{(\alpha, \alpha)}(x)
\Big\} \\
&= 
{n+2\alpha-1 \choose n} \frac{n+\alpha}{\alpha}
Q_n^{(\alpha, \alpha-1)}(x). 
\end{align*}
Putting $\alpha=\la+1/2$, we obtain the desired formula
 \eqref{eq:gegen2jacobi}. 
\end{proof}

The correlation kernel $\bK^{(d)}_n$ has the following simple form. 
\begin{prp} \label{prop:Kd}
For $u, v \in \SS^d$, $d \in \N$, and $n \in \N$, 
\begin{equation}
\bK^{(d)}_n(u, v)
= \Dn Q_n^{(d/2,d/2-1)}(u \cdot v). 
\label{eqn:K1_4}
\end{equation}
\end{prp}
\begin{proof}
Setting $\la = (d-1)/2$ in the formula 
\eqref{eq:gegen2jacobi}, we obtain \eqref{eqn:K1_4}. 
\end{proof}

In the next section, we use the following Mehler-Heine type
asymptotic formula for Jacobi polynomials. 
\begin{lmm}[Theorem 8.1.1.\cite{Sz}]\label{lem:jacobi}
For $\alpha, \beta \in \R$ and $z \in \C$, 
\[
 \lim_{n \to \infty} n^{-\alpha} 
P_n^{(\alpha, \beta)}(\cos \frac{z}{n})
= (z/2)^{-\alpha} J_{\alpha}(z),  
\]
and hence,  
\[
 \lim_{n \to \infty} Q_n^{(\alpha, \beta)}(\cos \frac{z}{n})
= \Gamma(\alpha+1) (z/2)^{-\alpha} J_{\alpha}(z). 
\]
The convergence takes place uniformly on every compact set
in $\C$. 
\end{lmm}

\section{Scaling limits of DPPs on $\SS^d$} \label{sec:infinite_DPP}
We consider the harmonic ensemble on $\SS^d$ and then 
define a DPP on the tangent space $T_p(\SS^d)$ at a point $p \in \SS^d$ 
as the pullback of points by the exponential map $\exp : T_p(\SS^d) \to \SS^d$. 
By rotation invariance, it suffices to look at 
the neighborhood of the `north pole' $\e_{d+1}$ on $\SS^d$ as
before. 
For $u \in \SS^d$, 
we use the polar coordinates 
$(\theta_1,\theta_2,\dots, \theta_d)$ in \eqref{eqn:coordinate1}
to write $\Omega(u):=(\Omega_1, \dots, \Omega_d) \in
\SS^{d-1}$ by setting 
\begin{align}
\Omega_1 &= \sin \theta_{d-1} \cdots \sin \theta_2 \sin \theta_1,
\nonumber\\
\Omega_k &= \sin \theta_{d-1} \cdots \sin \theta_k \cos \theta_{k-1},
\nonumber\\
\Omega_d &= \cos \theta_{d-1},
\nonumber\\
& \qquad \theta_1 \in [0, 2 \pi), \quad
\theta_k \in [0, \pi], \quad k=2, \dots, d-1.
\label{eqn:v-coordinate}
\end{align}
The tangent space $T_{\e_{d+1}}(\SS^d)$ is thus identified with 
\begin{equation}
T_{e_{d+1}}(\SS^d) \simeq \{\theta_d \Omega(u) \in \R^d : \theta_d \ge
0, u \in \SS^d\}, 
\label{eq:tangentaten} 
\end{equation}
and then the exponential map is given by 
\begin{equation}
T_{\e_{d+1}}(\SS^d) \ni (\theta_d \Omega_1,\dots, \theta_d \Omega_d) 
\mapsto 
 ((\sin \theta_d) \Omega_1,\dots, (\sin \theta_d) \Omega_d, 
 \cos \theta_d) \in \SS^d.  
\label{eq:tangentexp} 
\end{equation}
Let $\xi_n$ be the DPP on $\SS^d$ 
associated with the kernel $\bK^{(d)}_n(u, v)$ 
defined in the previous section. 
The number of points in $\xi_n$ is 
equal to $\Dn$ almost surely. 
For this DPP and fixed $\epsilon > 0$, we define 
   $\eta_n^{(\eps)}$ as the pullback of $\xi_n$ restricted
   to $\SS^d \cap B_{\epsilon}(\e_{d+1})$ by the 
       exponential map $\exp : T_{\e_{d+1}}(\SS^d) \to \SS^d$. 
We note that the scaling $S_{n}$ makes 
the density of points on $T_{\e_{d+1}}(\SS^d)$ of $O(1)$. 
For $u, v \in \SS^d$, we write 
$x = \theta_d \Omega(u), y = \vartheta_d \Omega(v) \in
T_{\e_{d+1}}(\SS^d)$  
as in \eqref{eq:tangentaten}.  
Since the pullback of $\mu_d$ on $\SS^d$ 
is locally equal to $\omega_d^{-1}$ times 
the Lebesgue measure $dx$ on $T_{e_{d+1}}(\SS^d) \simeq
\R^d$, it follows from \eqref{eq:transformation3} and 
\eqref{eq:transformation4} that 
the kernel of the scaled DPP
$S_n(\eta_n^{(\epsilon)})$ with respect to the Lebesgue measure $dx$
is 
\begin{equation}
K^{(d)}_{n,\eps}(x,y) 
= \frac{1}{n^d \omega_d} \bK_n^{(d)}(u_n, v_n) 
\trivial_{B_{\epsilon}}(u_n) 
\trivial_{B_{\epsilon}}(v_n), 
\label{eq:knleps}
\end{equation}
where $u_n = \exp(x/n)$ and 
$v_n = \exp(y/n)$. 

In this setting, we have the following limit formula for correlation kernels. 
\begin{prp}
\label{thm:asymK}
Assume that $(u^{(n)}, v^{(n)})_{n \in \N}$ 
is a sequence of pairs of points in $\SS^d$ such that
\[
u^{(n)} \cdot v^{(n)} = \cos \frac{r}{n}
\quad \mbox{with $0 < r < \infty$}.
\]
Then the limit
\begin{equation}
k^{(d)}(r) =
\lim_{n \to \infty} \frac{1}{n^d  \omega_d} \bK^{(d)}_{n}(u^{(n)}, v^{(n)})
\label{eqn:def_lim_K}
\end{equation}
exists and have the following expressions,
\begin{align}
k^{(d)}(r) 
&= 
\frac{J_{d/2}(r)}{(2 \pi r)^{d/2}}, 
\label{eqn:lim_K2} \\
&= \frac{1}{(2 \pi)^{d/2} r^{(d-2)/2}}
\int_0^1 s^{d/2} J_{(d-2)/2}(rs) ds
\label{eqn:lim_K}. 
\end{align}
\end{prp}
\begin{proof}
By \eqref{eqn:omega}, Proposition~\ref{prop:Kd} and Lemma~\ref{lem:jacobi}, we see that 
\[
\lim_{n \to \infty} \frac{1}{n^d \omega_d} \bK^{(d)}_{n} 
\big(u^{(n)}, v^{(n)}\big)
= \frac{2}{d!} \frac{\Gamma((d+1)/2)}{2 \pi^{(d+1)/2}}
 \Gamma\Big(\frac{d}{2}+1 \Big) \Big(\frac{2}{r}\Big)^{d/2} J_{d/2}(r)
= \left( \frac{1}{2\pi r} \right)^{d/2}
J_{d/2}(r).
\]
Here we used Legendre's duplication formula for the Gamma
 function (Eq.(5.5.5) in \cite{NIST10})
\[
\Gamma(z)\Gamma\Big(z+\frac{1}{2}\Big) = 2^{1-2z} \sqrt{\pi}
 \Gamma(2z). 
\]
The convergence takes place uniformly in any compacts in
 $\C$. Hence (\ref{eqn:lim_K2}) is proved.
If we use the integral formula (see, for instance, Eq.(10.22.1) in \cite{NIST10}),
\[
\int z^{\nu+1} J_{\nu}(z) dz = z^{\nu+1} J_{\nu+1}(z),
\]
we can derive (\ref{eqn:lim_K}) from (\ref{eqn:lim_K2}),
since
\[
\frac{1}{r^{(d+2)/2}} \int_0^r u^{(d-2)/2+1} J_{(d-2)/2}(u) du
= \int_0^1 s^{d/2} J_{(d-2)/2}(rs) ds.
\]
Thus the proof is complete. 
\end{proof}

This result implies that for each $d \in \N$ 
we obtain a limiting DPP on $\R^d$
such that it is uniform on $\R^d$ 
and the correlation kernel is given by
\begin{equation}
K^{(d)}(x, y) = k^{(d)}(|x-y|),
\quad x, y \in \R^d,
\label{eqn:bK1}
\end{equation}
where $k^{(d)}$ is defined by (\ref{eqn:lim_K2}). 
See also the formula \eqref{eq:smallk} below for odd $d$. 

Putting it all together, we have the following. 
\begin{thm}\label{thm:main}
The scaled point process $S_n(\eta_n^{(\eps)})$ converges 
weakly to the DPP on $T_{e_{d+1}}(\SS^d) \simeq \R^d$ 
associated with the kernel 
\begin{equation}
K^{(d)}(x,y) 
= \frac{1}{(2\pi |x-y|)^{d/2}} J_{d/2}(|x-y|)
\label{eq:knxy}
\end{equation}
and background measure $dx$. 
The kernel is also expressed in terms of Fourier transform
 as follows: 
\begin{equation}
K^{(d)}(x,y) 
= \frac{1}{(2\pi)^{d}}
\int_{\R^d} {\bf 1}_{B_1^{(d)}}(u) e^{\sqrt{-1} u \cdot (x-y)}
du, 
\label{eq:knxy2}
\end{equation}
where $B_1^{(d)}$ is the unit ball centered at the origin in $\R^d$. 
\end{thm}
\begin{proof}
First we assume that both $x, y \in \R^d$ are within a ball
 of radius less than the injectivity radius at $e_{d+1}$. 
Let $u_{n} = \exp(x/n)$ and $v_{n} = \exp(y/n)$. 
Then, by \eqref{eq:tangentaten} and \eqref{eq:tangentexp}, we see that  
\begin{align*}
 u_{n} \cdot v_{n} 
&= \sin \frac{\th_d}{n} \sin \frac{\vartheta_d}{n} 
\Omega(u_1) \cdot \Omega(v_1)+ \cos \frac{\th_d}{n} \cos \frac{\vartheta_d}{n} \\
&= \frac{\theta_d \vartheta_d}{n^2} 
\Omega(u_1) \cdot \Omega(v_1)+ 1 - \frac{1}{2} 
\frac{\theta_d^2 + \vartheta_d^2}{n^2} + o(\frac{1}{n^2}) \\
&= 1 - \frac{1}{2n^2} |x - y|^2 + o(\frac{1}{n^2}) \\
&= \cos \frac{|x - y|}{n} + o(\frac{1}{n^2}). 
\end{align*}
For general $x,y \in \R^d$, if necessary, 
by taking sufficiently large $M$ and considering $x/M$ and $y/M$
 instead of $x$ and $y$ themselves, we can show the same
 equality above. 
Therefore, from \eqref{eqn:K1}, \eqref{eq:knleps} and
 Proposition~\ref{thm:asymK}, we have \eqref{eq:knxy}. 
The weak convergence follows from the uniform convergence of
 the kernel on any compacts as mentioned in the proof of 
Proposition~\ref{thm:asymK}. 
\end{proof}

The limiting DPP associated with $K^{(d)}(x,y)$ is motion
invariant and the density of points is equal to $V_d / (2\pi)^d$,
where $V_d$ is the volume of the unit ball in $\R^{d}$. 
The kernel $K^{(d)}(x,y)$ is sometimes called the reproducing kernel for 
a generalized Paley-Wiener space. 
This kernel also appears in an asymptotic formula for Szeg\H{o}
kernel in the study of universality in \cite{Z01}. 

\begin{rem}
While the scaling was $S_{\sqrt{N}}$ in
 Section~\ref{sec:spherical}, 
here is $S_n$. 
The symbol $N$ in Section~\ref{sec:spherical} was used
 for the number of points on the sphere $\SS^2$. 
If we use the number of points 
$\Dn$ defined in \eqref{eqn:Dim} as parameter, then $S_n$ can also be
 written as $S_{(\Dn)^{1/d}}$. 
\end{rem}

\begin{rem}
The Bessel function $J_{m + 1/2}(x)$ with $m \in \N$
can be expressed as a linear combination of 
$\sin x$ and $\cos x$. 
The spherical Bessel function of the first kind is as
follows: 
\[
 j_m(x) := \sqrt{\frac{\pi}{2x}} J_{m + 1/2}(x). 
\]
Rayleigh's formula (Equation (10.51.3) in \cite{NIST10}) states that 
\[
 j_m(x) = (-x)^m \Big(\frac{1}{x} \frac{d}{dx} \Big)^m \ 
 \frac{\sin x}{x}, 
\]
from which it is easy to see that 
\begin{equation}
 k^{(d)}(r) = \Big(\frac{-1}{2 \pi r} \frac{d}{dr} \Big)^{(d-1)/2} \ 
 \frac{\sin r}{\pi r} \quad \text{for odd $d$}.  
\label{eq:smallk} 
\end{equation}
\end{rem}

\vskip 5mm
\textbf{Acknowledgments.} 
This work was supported by the Grant-in-Aid for Scientific
Research (C) (No.26400405), (B) (No.18H01124), and (S)
(No.16H06338) of Japan Society for the Promotion of Science.
One of the present authors (MK) thanks Christian Krattenthaler very much for his hospitality in Fakult\"{a}t f\"{u}r Mathematik, Universit\"{a}t Wien, where this study was started
on his sabbatical leave from Chuo University.
The authors express their gratitude to Michael Schlosser for
valuable discussion. 
The authors would also like to thank the referee for
several comments on our manuscript.


\vskip 1cm


\begin{thebibliography}{99}

\bibitem{AGR19} Abreu,~L.~D., Gr\"ochenig,~K. and
	Romero,~J.~L.,  
	Harmonic analysis in phase space and finite
	Weyl-Heisenberg ensembles, \textit{J. Stat. Phys.}
	\textbf{174} (2019), 1104--1136. 

\bibitem{AZ15}Alishahi,~K. and Zamani,~M., The
	spherical ensemble and uniform distribution of
	points on the sphere, 
	\textit{Electron. J. Probab.} \textbf{20} (2015), no. 23,
	1--27. 

\bibitem{A50} Aronszajn,~N., Theory of reproducing
	kernels, \textit{Trans. Amer. Math. Soc.} \textbf{68} (1950),
	337--404.  

\bibitem{BMOC16} Beltr\'an,~C., Marzo,~J. and
	Ortega-Cerd\`a., 
	Energy and discrepancy of rotationally invariant
	determinantal point processes in high dimensional
	spheres, \textit{J. Complexity} \textbf{37} (2016), 76--109. 

\bibitem{BQ17} Bufetov,~A.~I. and Qiu,~Y., Determinantal point
	processes associated with Hilbert spaces of
	holomorphic
	functions, \textit{Commun. Math. Phys.} \textbf{351} (2017) 
	1--44. 

 \bibitem{Ca81} Caillol,~J.~M., Exact results for a
	 two-dimensional one-component plasma on a sphere, 
	 \textit{Journal de Physique Lettres} \textbf{42} (1981),
	 245--247. 

\bibitem{For10} Forrester,~P.~J.,  
{\it Log-gases and Random Matrices}, London Math. Soc. Monographs,
Princeton University Press, Princeton (2010).  

\bibitem{K86} Kallenberg,~O., \textit{Random measures},
	Fourth ed. Akademie-Verlag, Berlin; Academic Press,
	Inc., London, 1986. 

\bibitem{Nom18}
	Nomura,~T., \textit{Spherical Harmonics and Group Representations}.
	Nihon-Hyoronsya, 2018 (Japanese). 

\bibitem{NIST10}
Olver,~F.~W.~J., Lozier,~D.~W., Boisvert,~R.~F. and Clark,~C.~W.(ed): {\it NIST Handbook of Mathematical Functions}. 
U.S. Department of Commerce, 
National Institute of Standards and Technology, 
Washington, DC, 
Cambridge: Cambridge University Press, 2010; 
available at \url{http://dlmf.nist.gov}

 \bibitem{HKPV} Hough,~J.~B., Krishnapur,~M., Peres,~Y. and
 	Vir\'{a}g,~B., \textit{Zeros of Gaussian analytic functions
 	and determinantal point processes}. University
 	Lecture Series, 51. American Mathematical Society,
 	Providence, RI, 2009. 
\bibitem{K09} Krishnapur,~M., From random matrices to random
	analytic functions. \textit{Ann. Probab.} \textbf{37} (2009),
	314--346.  
\bibitem{M75} Macchi,~O., The coincidence approach to
	stochastic point processes.  
	\textit{Adv. Appl. Prob.} \textbf{7} (1975), 83--122. 
\bibitem{S16} Shirai,~T.,  
	Ginibre-type point processes and their asymptotic
	behavior. 
	\textit{Journal of the Mathematical Society of
	Japan}, \textbf{67} (2015), 763--787. 
\bibitem{ST0} Shirai,~T. and Takahashi,~Y., 
	Fermion process and Fredholm determinant, \textit{Proceedings of the
	Second ISAAC Congress}, Vol. 1 (Fukuoka, 1999), 15--23, 
	Kluwer Acad. Publ., 2000. 
\bibitem{ST1} Shirai,~T. and Takahashi,~Y., 
	Random point fields associated with certain Fredholm determinants I: 
	fermion, Poisson and boson point processes, 
	\textit{J. Funct. Anal.} \textbf{205} (2003), 414--463. 
\bibitem{S} Soshnikov,~A., Determinantal random point
	fields. \textit{Translation in Russian Math. Surveys}
	\textbf{55} (2000), 923--975.  
\bibitem{Sz} Szeg\H{o},~G., \textit{Orthogonal Polynomials}. American
	Mathematical Society, Colloquium Publications \textbf{23}, 1939. 

\bibitem{Z01} Zelditch,~S., From random polynomials to
	symplectic geometry, \textit{XIIIth International
	Congress on Mathematical Physics (London, 2000)},
	367--376, Int. Press, Boston, MA, 2001.  


\end{thebibliography}
\end{document}